\documentclass{amsart}[10pt,a4paper,twoside]
\usepackage{fancyhdr}
\usepackage[T1]{fontenc}
\usepackage[latin1]{inputenc}
\usepackage[english]{babel}
\usepackage{times}
\usepackage{
  amsmath,
  amssymb,
  enumerate,
  bbm,
  textcomp,
  appendix,
  tabularx
}
\usepackage{float}
\usepackage{tikz}
\usepackage{tikz-cd}
\usetikzlibrary{patterns}
\usepackage{epsfig}
\usepackage[all]{xy}
\usepackage{hyperref}
\usepackage{verbatim}
\usetikzlibrary{decorations.pathreplacing}
\theoremstyle{plain}
\usepackage{lscape}
\theoremstyle{plain}
\newtheorem{thm}{Theorem}[section]
\newtheorem{prop}[thm]{Proposition}
\newtheorem{cor}[thm]{Corollary}
\newtheorem{lem}[thm]{Lemma}
\theoremstyle{definition}
\newtheorem{ex}[thm]{Example}
\newtheorem{conj}[thm]{Conjecture}

\newcommand{\cla}{\mathcal{L}_{\mathnormal{A}}}
\newcommand{\cra}{\mathcal{R}_{\mathnormal{A}}}
\newcommand{\ind}{\mathrm{ind}\,}

\newcommand{\add}{\mathrm{add}\, }

\newcommand{\End}{\mathrm{End}}

\renewcommand{\ker}{\operatorname{Ker} }
\newcommand{\coker}{\operatorname{Coker} }
\newcommand{\im}{\operatorname{Im} }
\newcommand{\md}{\mathrm{mod}\, }
\newcommand{\tor}{\operatorname{Tor}}
\renewcommand{\hom}{\operatorname{Hom}}

\newcommand{\dimp}{\mathrm{pd}  }
\newcommand{\di}{\mathrm{id} }
\newcommand{\dimgl}{\mathrm{gl.dim} }
\newcommand{\maximo}{\textit{max}}
\newcommand{\ext}{\mathrm{Ext}}
\newcommand{\gd}{\mathrm{gl.dim}}
\newcommand{\rar}{\rightarrow}
\newcommand{\xrar}{\xrightarrow}

\newcommand{\cal}{\mathcal}

\numberwithin{equation}{section}

\begin{document}

\title[$(m,n)$-Quasitilted and $(m,n)$-almost hereditary algebras]{$(m,n)$-Quasitilted and $(m,n)$-almost hereditary algebras}


\author{Edson Ribeiro Alvares}
\author{Diane Castonguay}
\author{Patrick Le Meur}
\author{Tanise Carnieri Pierin}

\address[Edson Ribeiro Alvares]{Centro Polit\'ecnico, Departamento de Matem\'atica,
  Universidade Federal do Paran\'a, CP019081, Jardim das Americas,
  Curitiba-PR, 81531-990, Brazil
}

\email{rolo1rolo@gmail.com, rolo@ufpr.br}

\address[Diane Castonguay]{Instituto de Inform\'atica, Universidade Federal de Goi\'as, Campus II - Samambaia, CEP: 74001-970, Goi\^ania, Brazil}
\email{diane@inf.ufg.br}

\address[Patrick Le Meur]{Universit\'e Paris Diderot, Sorbonne Paris Cit\'e, Institut de
  Math\'ematiques de Jussieu-Paris Rive Gauche, UMR 7586, CNRS,
  Sorbonne Universit\'es, UMPC Univ. Paris 06, F-75013, Paris, France}
\email{patrick.le-meur@imj-prg.fr}

\address[Tanise Carnieri Pierin]{Centro Polit\'ecnico, Departamento de Matem\'atica,
  Universidade Federal do Paran\'a, CP019081, Jardim das Americas,
  Curitiba-PR, 81531-990, Brazil
}
\email{tanpierin@gmail.com, tanise@ufpr.br}

\thanks{The first named author acknowledges support from DMAT-UFPR and
  CNPq - Universal 477880/2012-6}

\thanks{The fourth named author acknowledges support from CNPq and CAPES}

\date{\today}

\begin{abstract}
  Motivated by the study of $(m,n)$-quasitilted algebras, which are
  the piecewise hereditary algebras obtained from quasitilted algebras
  of global dimension two by a sequence of (co)tiltings involving
  $n-1$ tilting modules and $m-1$ cotilting modules, we introduce
  $(m,n)$-almost hereditary algebras. These are the algebras with global
  dimension $m+n$ and such that any indecomposable module has
  projective dimension at most $m$, or else injective dimension at
  most $n$. We relate these two classes of algebras, among which
  $(m,1)$-almost hereditary ones play a special role. For these, we
  prove that any indecomposable module lies in the right part of the
  module category, or else in an $m$-analog of the left part. This is
  based on the more general study of algebras the module categories of
  which admit a torsion-free subcategory such that any indecomposable
  module lies in that subcategory, or else has injective dimension at
  most $n$.
\end{abstract}

\maketitle

\setcounter{tocdepth}{1}
\tableofcontents

\section*{Introduction}

Quasitilted algebras were defined in \cite{MR1327209} as
the opposite algebras of endomorphism algebras of tilting objects
of $\hom$-finite, Krull-Schmidt hereditary abelian categories. These
algebras feature several properties and characterisations which
explain their relevance. First, by \cite[Chapter II, Theorem
2.3]{MR1327209}, an Artin algebra over an Artin ring is quasitilted
if and only if it has global dimension at most two and, for any
indecomposable module $X$,
\begin{equation}
  \label{eq:9}
  \dimp\,X\leqslant 1\ \text{or else}\ \di\,X\leqslant 1\,.
\end{equation}
Next, for any quasititled algebra $A$, the following decomposition
holds (see \cite[Chapter II, Proposition 1.6]{MR1327209})
\begin{equation}
  \label{eq:10}
  \ind A = \mathcal L_A \cup \mathcal R_A,
\end{equation}
where $\mathcal L_A$ and $\mathcal R_A$ are the left and right parts
of the module category of $A$, respectively. These subcategories are
efficient tools to classify algebras and study their representation
theory. In particular (see \cite[Theorem 1.14]{MR1327209}),
\begin{equation}
  \label{eq:11}
  \text{$A$ is quasitilted if and only if $A\in \add \mathcal L_A$;}
\end{equation}
we refer the reader to \cite{MR2147013} for a survey on these
subcategories.  Finally, by general theory of derived equivalences
arising from tilting complexes, they are piecewise hereditary, that
is, the bounded derived category of their module categories are
triangle equivalent to those of hereditary abelian categories.

Recall that, for any piecewise hereditary algebra, there is a sequence
of algebras with first term a quasitilted algebra, with last term the
given algebra and such that each algebra of the sequence is the
opposite algebra of the endomorphism algebra of a splitting tilting or
cotilting module over the preceding algebra (see
\cite{zbMATH00938523}, or Section~\ref{sec:m-n-quasitilted-1}
for a reminder). In many examples, that given algebra often has
homological properties like (\ref{eq:9}) depending on the number of
tilting modules and the number of cotilting modules involved in the
sequence.

\medskip

Given positive integers $m$ and $n$, we introduce $(m,n)$-quasitilted
algebras, defined as the algebras for which there is a sequence as
above such that the number of involved cotilting modules is $m-1$, the
number of involved tilting modules is $n-1$ and the global dimensions
of the algebras of the sequence increase strictly. The aim of this
article is to show that these algebras have homological properties
like (\ref{eq:9}) and to derive consequences relative to their
representation theory in terms of specific subcategories, like
(\ref{eq:10}).

For this purpose, we define the $(m,n)$-almost hereditary algebras as the
algebras with global dimension $m+n$ and such that, for any
indecomposable module $X$,
\begin{equation}
  \label{eq:12}
  \dimp\,X\leqslant m \ \text{or else}\ \di\,X\leqslant n\,.
\end{equation}
Note that (\ref{eq:12}) alone implies that the global dimension is at
most $m+n+1$ (see Lemma~\ref{sec:m-n-quasitilted-2}).

Our results establish that any $(m,n)$-quasitilted algebra is
$(m,n)$-almost hereditary (Corollary~\ref{sec:m-n-quasitilted-3});
although this is an equivalence when $(m,n)=(1,1)$, the converse
implication does not hold in general. Moreover, it appears that any
$(m,n)$-quasitilted algebra is $(m+n-1,1)$- or else $(1,m+n-1)$-almost
hereditary (Lemma~\ref{sec:m-n-quasitilted-4}). This shows the
relevance of $(m,1)$- and $(1,m)$-almost hereditary algebras among the
ones considered previously. For these algebras, we prove in
Theorem~\ref{principal} that, if $A$ is $(m,1)$-almost hereditary,
then
\begin{equation}
  \label{eq:14}
  \ind A = \mathcal
  L^m_A \cup \mathcal R_A,
\end{equation}
where $\mathcal L_A^m$ denotes the class of indecomposable modules
such that the predecessors in $\ind A$ of which have projective dimension at
most $m$. This result is obtained as a consequence of a more general
one: we prove in Proposition~\ref{sec:suggestion-section-4-5} that if
$\mathcal C$ is a torsion-free class in the module category of $A$
such that any indecomposable module lies in $\mathcal C$ or else has
injective dimension at most $n$, then
\begin{equation}
  \label{eq:13}
  \ind A = \mathcal L_{\mathcal C} \cup \mathcal R^m_A,
\end{equation}
where the definition of $\mathcal R_A^m$ is dual to that of
$\mathcal L_A^m$ and the definition of $\mathcal L_{\mathcal C}$ is
given in \ref{sec:algebras-with-small}. These results raise the
question to determine when an algebra is $(m,1)$-almost hereditary. We
give a partial answer to this question in
Proposition~\ref{sec:suggestion-section-4-7} (compare with (\ref{eq:11})),
\begin{center}
  $A$ lies $\mathcal L_A^m$ and has global dimension $m+1$
  $\Rightarrow$ $A$ is $(m,1)$-almost hereditary.
\end{center}

\medskip

The article is therefore organised as follows. Base material is
collected in
Section~\ref{sec:preliminaries}. Section~\ref{sec:m-n-quasitilted-1}
introduces $(m,n)$-quasitilted algebras and $(m,n)$-almost hereditary
ones, and relates them. Section~\ref{sec:m-1-almost} studies
$(m,1)$-almost hereditary algebras in terms of $\mathcal L_A^m$ and
$\mathcal R_A$. Finally, Section~\ref{sec:one-point-extensions}
studies the behaviour of $(m,1)$-almost hereditary algebras under
taking one point extensions.

\medskip

Throughout the text, $k$ denotes an
Artin commutative ring and $A$ denotes an Artin $k$-algebra.

\section{Preliminaries}
\label{sec:preliminaries}
We denote by $\md A$ the category of finitely generated left
$A$-modules and by $\ind A$ a full subcategory consisting of exactly
one representative from each isomorphism class of indecomposable
$A$-modules. For a subcategory ${\cal C}$ of $\md A$ we write
$M \in {\cal C}$ to express that $M$ is an object in ${\cal C}$.

If $T \in \md A$, then $\add T$ denotes the full subcategory of
$\md A$ whose objects are the direct sums of direct summands of $T$.
Given an $A$-module $M$, we denote by $\dimp_A M$ and $\di_A M$,
respectively, its projective and injective dimensions. The global
dimension of $A$ is denoted by $\dimgl A$.

The space of morphisms from an object $X$ to an object $Y$ in $\mathcal
D^b(\md A)$ is denoted by $\hom(X,Y)$.
Whenever $T\in \mathcal D^b(\md A)$ is a tilting complex, that is,
there is no nonzero morphism $T\to T[i]$ in $\mathcal D^b(\md A)$ for
all $i\in \mathbb Z \backslash\{0\}$ and $\mathcal D^b(\md A)$ is the
smallest triangulated subcategory of $\mathcal D^b(\md A)$ containing
$T$ and stable under taking direct summands, the module categories of
$\md A$ and $\md B$ have equivalent bounded derived categories, where
$B$ is the endomorphism algebra of $T$ in $\mathcal D^b(\md A)$. In
such a situation, $\md B$ is identified with the following subcategory
of $\mathcal D^b(\md A)$,
\[
\{X \in \mathcal D^b(\md A)\ |\ (\forall i\in \mathbb
Z\backslash\{0\})\ \hom(T,X[i])=0\}\,.
\]
In particular, whenever $X,Y\in \md A\,\cap\, \md B$ and $i\in \mathbb Z$,
then $\ext_A^i(X,Y)$ is naturally identified with $\ext_B^i(X,Y)$.

For further background on the representation theory of $A$, we refer
the reader to \cite{MR1314422}, \cite{MR2197389}, \cite{MR935124}.

\subsection{Paths, left and right parts} Given $M, N \in \ind A$, a {\it path} from $M$ to $N$ (denoted by $M \rightsquigarrow N$) is a sequence of nonzero morphisms 
$$(*) \ \ M = M_0 \stackrel{f_1}\rightarrow M_1 \rightarrow \cdots \stackrel{f_t}\rightarrow M_t = N$$
where $M_i \in \ind  A$ for all $i$. In this case $N$ is called {\it successor} of $M$ and $M$ {\it predecessor} of $N$.

Following \cite{MR1939113}, the {\it left part} $\cla$ of $\md A$ is
the full subcategory whose objects are those $M \in \ind A$ such that
every predecessor of $M$ in $\ind A$ has projective dimension at most
one. Clearly $\cla$ is closed under predecessors. The {\it right part}
$\cra$ is defined dually and has dual properties.

\subsection{A basic fact on short exact sequences} 
The following lemma is used several times in this article (see
\cite[Lemma 1.2, part (ii)]{MR2413349} for a similar statement).
\begin{lem}
  Let $0 \rar X \xrar{f}  E \xrar{g} Y \rar 0$ be a non-split exact sequence.
  \begin{itemize}
  \item[$(a)$] If $X$ is indecomposable then each coordinate morphism of $g$ is nonzero.
  \item[$(b)$] If $Y$ is indecomposable then each coordinate morphism of $f$ is nonzero.
  \end{itemize}
\end{lem}
\begin{proof}
  We only  prove $(a)$ since the proof of $(b)$ can be obtained dually. Suppose that $E=E_1\oplus E_2$, where $E_1$ is indecomposable, and $g_1 = 0$. Consider the following commutative diagram
   \begin{center}
     \begin{tikzcd}[ampersand replacement=\&]
       \& 0 \ar{d} \& 0 \ar{d} \&\& \\
       \& E_1 \ar{r}{1} \ar{d} \& E_1 \ar{d}{\left(\begin{smallmatrix} 1 \\ 0 \end{smallmatrix}\right)} \&\& \\ 
       0 \ar{r} \& X \ar{r}{\left(\begin{smallmatrix} f_{1} \\ f_{2} \end{smallmatrix}\right)} \ar{d} \& E_1 \oplus E_2 \ar{r}{\left(\begin{smallmatrix} 0 \ g_{2} \end{smallmatrix}\right)} \ar{d}{\left(\begin{smallmatrix} 0 \ 1 \end{smallmatrix}\right)} \& Y \ar{r} \ar{d}{1} \& 0 \\
       0 \ar{r} \& \ker g_2 \ar{r} \ar{d} \& E_2 \ar{r}{g_2} \ar{d} \& Y \ar{r} \& 0\\
       \& 0 \& 0. \& \& 
     \end{tikzcd}
   \end{center}
  It follows that  $f_1$ is a split epimorphism and, because $X$ is indecomposable, $f_1$ is an isomorphism. Then the exact sequence $0 \rar X \rar E_1 \oplus E_2  \rar Y \rar 0$ splits, which is a contradiction.
\end{proof}

\subsection{A short review of tilting theory}

Recall that $_{A}T \in \md A$ is called a {\it tilting module} provided the following three conditions are satisfied:
\begin{enumerate}
\item $\dimp _{A} T\leq 1$,
\item  $\ext^{1}_{A}(T, T) = 0$
\item the number of pairwise non-isomorphic indecomposable direct summands
  of $T$ equals the rank of the Grothendieck group, $K_{0}(A)$.
\end{enumerate} 

Given a tilting module $_{A}T$, let $B = (\End_{A} T)^{\rm
  op}$.
Recall that $_{A}T_{B}$ induces torsion pairs
$(\mathcal{T}(T), \mathcal{F}(T))$ on $\md A$ and
$(\mathcal{X}(T), \mathcal{Y}(T))$ on $\md B$, where
${\cal T}(T) = \{ X \in \md A; \ext^{1}_{A}(T, X) = 0 \}$,
${\cal F}(T) = \{X \in \md A; \hom_{A}(T, X) = 0 \}$,
${\cal X}(T) = \{ X \in \md B; T \otimes_{B} X = 0 \}$ and
${\cal Y}(T) = \{X \in \md B; \tor_{1}^{B}(T,X) = 0 \}$. Due to the
Brenner - Butler Theorem (\cite{MR675063}), $\md A$ and $\md B$ are
related as follows: restriction of functor
$\hom_{A}(T,-): \md A \rightarrow \md B$ to $\cal T(T)$ is an
equivalence ${\cal T }(T) \rightarrow {\cal Y}(T)$ and restriction of
$\ext_{\Lambda}^{1}(T,-)$ to ${\mathcal{F}}(T)$ gives rise to an
equivalence ${\cal F}(T) \rar {\cal X}(T)$. A tilting module $_{A}T$
is called {\it splitting} if the torsion pair
$(\mathcal{X}(T), \mathcal{Y}(T))$ on $\md B$ splits, that is, if each
indecomposable $B$-module lies in $\mathcal{X}(T)$, or else in
$\mathcal{Y}(T)$. According to Hoshino \cite{zbMATH03797988}, if
$_{A}T$ is a tilting module, then $_{A}T$ is splitting if and only if
$\di_{A} X \leq 1$ for every $X \in {\cal F}(T)$.

Dually, it is possible to define a {\it cotilting} module. We will use
but not state explicitly the dual results and properties which hold
true for cotilting modules. For further definitions and results on
tilting theory, we refer the reader to \cite{MR675063},
\cite{MR774589}.

\section{$(m,n)$-Quasitilted algebras}
\label{sec:m-n-quasitilted-1}

The purpose of the section is to relate $(m,n)$-quasitilted algebras
to $(m,n)$-almost hereditary algebras.  We say that $A$ is {\it
  $(m,n)$-quasitilted} if there exists a sequence of triples
$(A_i, T_i, A_{i+1} = (\End_{A_i} T_i)^{\rm op})$ such that: $A_0$ is a
quasitilted algebra of global dimension two; $A = A_{m+n-2}$; each
$T_i$ is a \emph{stair} splitting tilting or cotilting $A_i$-module,
that is a splitting tilting or cotilting module with the property that
$\dimgl A_i < \dimgl A_{i+1}$; and $n-1$ modules among the $T_i$ are
tilting whereas the $m-1$ remaining ones are cotilting. Any
$(m,n)$-quasitilted algebra is piecewise hereditary. Actually, it is
proved in \cite{zbMATH00938523} that any piecewise hereditary algebra
$A$ may be obtained from some quasitilted algebra by a sequence of
tilting or cotilting processes, where each involved (co)tilting module
is splitting. And, when $k$ is a field and $A$ is derived equivalent
to a finite dimensional hereditary algebra, then $A$ may be obtained
from hereditary algebra by a sequence of tilting processes, where each
involved tilting module is splitting (see \cite{MR916069}).

By definition, the $(1,1)$-quasitilted algebras are the quasitilted
algebras of global dimension two; accordingly a $(1,1)$-quasitilted
algebra $A$ satisfies {\it (i)} $\gd A =2$ and {\it (ii)}
$\dimp_A X \leq 1$ or else $\di_A X \leq 1$, for each indecomposable
$A$-module $X$ (\cite{MR1327209}). Below, we prove that any
$(m,n)$-quasitilted algebra $A$ has the two following homological
properties.
\begin{enumerate}
\item[$(Q1)$] $\gd A = m+n$;
\item[$(Q2)$] for each indecomposable $A$-module $X$, $\dimp_A X \leq m$ or else $\di_A X \leq n$.
\end{enumerate}
We call {\it $(m,n)$-almost hereditary} an algebra
which satisfies $(Q1)$ and $(Q2)$, where $m$ and $n$ are positive
integers.  Note that condition $(Q2)$ is not a consequence of
condition $(Q1)$ as shown in the following example.
\begin{ex}
  \label{sec:m-n-quasitilted-5}
  Let $A$ be the path algebra of the quiver
  \[
  \xymatrix@R=5pt{
    1 \ar[rr]
    \ar[dr]_{\alpha} & & 2 \\ & 3 \ar[ur]_{\beta} &
  }
  \]
  bound by
  the relation $\beta \alpha = 0$; the indecomposable $A$-module
  $\begin{smallmatrix} 2 \\ 1 \end{smallmatrix}$ has projective and
  injective dimension equal to two, while global dimension of $A$ is
  two.
\end{ex}
\label{sec:m-n-quasitilted-6}
Note also that condition $(Q1)$ cannot be obtained from condition
$(Q2)$ as shown in the following example.
\begin{ex}
  Consider the radical square zero algebra $A$
  given by the quiver
  \[
  1\to 2 \to \cdots \to m+n+2\,,
  \]
  where $m$, $n \geq 1$. In this case, $\gd A = m+n+1$, whereas each
  indecomposable $A$-module has projective dimension at most $m$ or
  injective dimension at most $n$.  Of course, $A$ is a $(a,b)$-almost
  hereditary algebra for all positive integers $a,b$ such that
  $a+b=m+n+1$.
\end{ex}
In fact, it is possible to verify that an algebra which
satisfies condition $(Q2)$ and with global dimension greater than
$m+n-1$ is a $(m,n)$-almost hereditary or a $(m+1,n)$-almost
hereditary algebra, as stated in the following lemma.
\begin{lem}
  \label{sec:m-n-quasitilted-2}
  Let $m$ and $n$ be positive integers and let $A$ be an algebra such
  that $\dimp_A X \leq m$ or $\di_A X \leq n$ for each indecomposable
  $A$-module $X$. Then $\gd A \leq m+n+1$. In particular, if
  $\gd A > m+n-1$ then $A$ is a $(m,n)$-almost hereditary or a
  $(m+1,n)$-almost hereditary algebra.
\end{lem}
\begin{proof}
  Let $M\in \ind A$. Then, any indecomposable direct summand of
  $\Omega^{n+1}M$ has injective dimension at least $n+1$, and hence
  has projective dimension at most $m$. Therefore,
  $\dimp_A\,\Omega^{n+1}M\leqslant m$. Accordingly,
  $\dimp_A\,M\leqslant m+n+1$.
\end{proof}

A first step to relate $(m,n)$-quasitilted algebras to
$(m,n)$-almost hereditary ones is to investigate how the former
behave under tilting (or, dually, under cotilting). The relationship
is then obtained as a corollary.
\begin{thm}
  \label{sec:m-n-quasitilted}
  Let $m,n$ be positive integers. Let $A$ be an algebra satisfying
  $(Q2)$. Let $T$ be a splitting tilting $A$-module. Denote
  $(\End_A T)^{\rm op}$
  by $B$. Then, any indecomposable $B$-module has projective dimension
  at most $m$, or else injective dimension at most $n+1$. In particular,
  if $A$ is $(m,n)$-almost hereditary and $T$ is stair, then $B$ is
  $(m,n+1)$-almost hereditary.
\end{thm}
\begin{proof}
  We only prove the first statement because the second one follows
  from the first one and from the definition of stair tilting modules.
  For all subcategories $\mathcal A$, $\mathcal B$ of $\md A$ and
  $i\in \mathbb Z$, denote by $\ext^i_A(\mathcal A,\mathcal B)$ the
  collection $\{\ext^i_A(X,Y); X \in \mathcal A, Y \in \mathcal B\}$. 
  Also, for a full subcategory $\mathcal A$ of $\md A$, we denote
  ${\rm sup}\,\{\di_A X; X\in \mathcal A\}$ by $\di_A\,\mathcal
  A$.
  First, $\di_B\,\mathcal X(T) \leqslant 1$. Indeed, using
  that $\di_A\,\mathcal F(T) \leqslant 1$, we have
  \[
  \begin{array}{rcl}
    \ext^2_B(\mathcal X(T),\mathcal X(T))
    & = &
          \ext^2_A(\mathcal F(T),\mathcal F(T))=0 \\
    \ext^1_B(\mathcal Y(T),\mathcal X(T))
    & = &
          \ext^2_A(\mathcal T(T),\mathcal F(T)) =0\,.
  \end{array}
  \]
  Next, consider the decomposition
  \[
  (\ind \mathcal Y(T)=)\ind \mathcal T(T) = \mathcal C_1 \cup \mathcal C_2\,,
  \]
  where $\mathcal C_1=\{X\in \ind \mathcal T(T); \dimp_A X\leqslant
  m\}$ and $\mathcal C_2 = \{X\in \ind \mathcal T(T);
  \di_A X\leqslant n\}$. 
  Since
  \[
  \begin{array}{rcl}
    \ext^{m+1}_B(\mathcal C_1,\mathcal Y(T))
    & = &
          \ext^{m+1}_A(\mathcal C_1,\mathcal T(T))=0 \\
    \ext^{m+1}_B(\mathcal C_1,\mathcal X(T))
    & = &
          \ext^{m+2}_A(\mathcal C_1,\mathcal F(T))=0\,,
  \end{array}
  \]
  it follows that $\dimp_B\,\mathcal C_1\leqslant m$. Also, $\di_B\,\mathcal C_2\leqslant n+1$ because
  \[
  \begin{array}{rcl}
    \ext^{n+2}_B(\mathcal X(T), \mathcal C_2)
    & = &
          \ext^{n+1}_A(\mathcal F(T),\mathcal C_2) = 0 \\
    \ext^{n+2}_B(\mathcal Y(T),\mathcal C_2)
    & = &
          \ext^{n+2}_A(\mathcal T(T),\mathcal C_2) = 0\,.
  \end{array}
  \]
  This proves the theorem.
\end{proof}

Now, here is the announced relationship between $(m,n)$-quasitilted
algebras and $(m,n)$-almost hereditary ones.
\begin{cor}
  \label{sec:m-n-quasitilted-3}
  If $A$ is a $(m,n)$-quasitilted algebra, then $A$ is $(m,n)$-almost hereditary.
\end{cor}
\begin{proof}
  Let $A$ be a $(m,n)$-quasitilted algebra. By definition, there
  exists a sequence of triples $(A_i, T_i, A_{i+1} = (\End_{A_i}
  T_i)^{\rm op})$
  such that $A_0$ is strict quasitilted, each $T_i$ is a stair and
  splitting tilting or cotilting $A_i$-module and $A = A_{m+n-2}$,
  where $m-1$ and $n-1$ are the numbers of cotilting and tilting
  processes, respectively. As mentioned before, $A_0$ is
  $(1,1)$-almost hereditary and therefore, according to
  Theorem~\ref{sec:m-n-quasitilted}, $A$ is $(m,n)$-almost hereditary.
\end{proof}

\begin{ex}
Let $A$ be the path $k$-algebra given by the bound quiver $1 \xrightarrow{\alpha_1} 2 \xrightarrow{\alpha_2} \cdots \xrightarrow{\alpha_{7}} 8$, with relations $\alpha_{i+4} \cdots \alpha_{i} = 0$, where $1\leq i \leq 3$. Consider the tilting $A$-module $T =P_4 \oplus T_2\oplus T_3\oplus T_4 \oplus S_4\oplus P_3 \oplus P_2 \oplus P_1$. Note that $B =  (\End{}_{A} T)^{\rm op}$ is a tilted algebra, which yields that $A$ is $(1,2)$-quasitilted, since  $A \simeq \End T_B$. According to Corollary \ref{sec:m-n-quasitilted-3}, $A$ is also $(1,2)$-almost hereditary. It is an easy verification that ${\cal L}_A$ and ${\cal R}_A^2$ consist, respectively, of those modules in horizontal and vertical lines patterned areas in the illustration below. We point out that ${\cal L}_A \cup {\cal R}_A \subsetneq \ind A$, since $M$ has projective and injective dimensions equal to $2$.

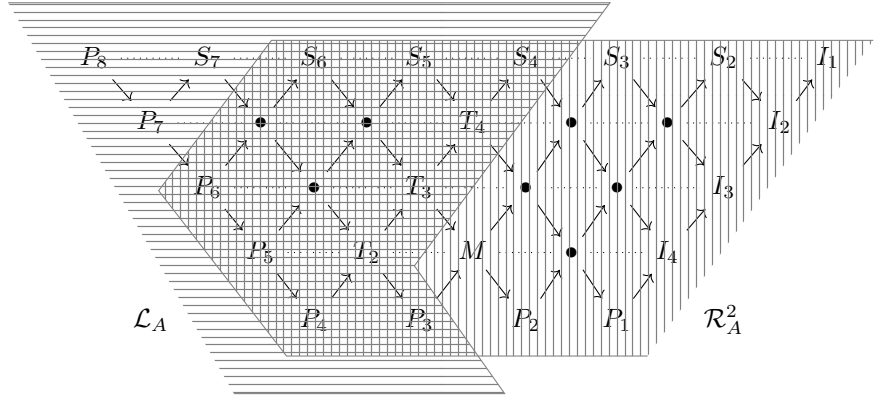
\begin{center}
\begin{tikzpicture}[baseline= (a).base]
\node (a) at (-20,0){

\begin{tikzcd}[row sep=0.3cm, column sep=0.01cm]
P_8 \ar[rr, dash, dotted] \ar[dr] && S_7\ar[rr, dash, dotted] \ar[dr] && S_6 \ar[rr, dash, dotted] \ar[dr] && S_5 \ar[rr, dash, dotted] \ar[dr] && S_4 \ar[rr, dash, dotted] \ar[dr] && S_3 \ar[rr, dash, dotted] \ar[dr] && S_2 \ar[rr, dash, dotted] \ar[dr] && I_1 \\
&  P_7 \ar[dr] \ar[ur] \ar[rr, dash, dotted] && \bullet \ar[dr] \ar[ur] \ar[rr, dash, dotted] && \bullet \ar[dr] \ar[ur] \ar[rr, dash, dotted] && T_4 \ar[dr] \ar[ur] \ar[rr, dash, dotted] && \bullet\ar[dr] \ar[ur] \ar[rr, dash, dotted] && \bullet \ar[dr] \ar[ur] \ar[rr, dash, dotted] && I_2 \ar[ur] & \\
&& P_6\ar[dr] \ar[ur] \ar[rr, dash, dotted] && \bullet \ar[dr] \ar[ur] \ar[rr, dash, dotted] && T_3 \ar[dr] \ar[ur] \ar[rr, dash, dotted] && \bullet \ar[dr] \ar[ur] \ar[rr, dash, dotted] && \bullet \ar[dr] \ar[ur] \ar[rr, dash, dotted] && I_3 \ar[ur]  && \\
&&& P_5 \ar[dr] \ar[ur] \ar[rr, dash, dotted] && T_2 \ar[dr] \ar[ur] \ar[rr, dash, dotted] && M \ar[dr] \ar[ur] \ar[rr, dash, dotted] && \bullet \ar[dr] \ar[ur] \ar[rr, dash, dotted] && I_4 \ar[ur]  &&& \\
&{\cal L}_A &  & &P_4 \ar[ur]  && P_3 \ar[ur]  && P_2 \ar[ur] && P_1 \ar[ur] && {\cal R}_A^2&&   
\end{tikzcd}
};
\draw[pattern=vertical lines, pattern color=gray, draw=gray] (-17.5, -2.2) -- (-22.3,-2.2) -- (-24,0) -- (-22.5,2) -- (-14.5,2) ;
\draw[pattern=horizontal lines, pattern color=gray, draw=gray]  (-26,2.5) -- (-18, 2.5) -- (-20.6,-1) -- (-19.4,-2.7) -- (-23,-2.7);
\end{tikzpicture}

\end{center}
\end{ex}

Note that the converse of the previous corollary is not true in
general. Here is a counter-example.
\begin{ex}
  Let $A$ be the path $k$-algebra given by the bound quiver
  $1 \xrightarrow{\alpha_1}  2 \xrightarrow{\alpha_2}  \cdots \xrightarrow{\alpha_{10}}  11 \xrightarrow{\alpha_{11}}  12,
  $
  with relations  $\alpha_{i+6} \cdots \alpha_{i} = 0$, where $1 \leq i \leq 5$.
  It follows from  \cite{MR2736030} that $A$ is not piecewise
  hereditary algebra, which entails that $A$ is not
  $(1,2)$- nor $(2,1)$-quasitilted (or even $(m,n)$-quasitilted for any $m$ and
  $n$). However, it is easily seen that $A$ is $(1,2)$- and $(2,1)$-almost
  hereditary.
\end{ex}

Later, the article concentrates on $(m,1)$- and $(1,m)$-almost
hereditary algebras. The reason is the following result which states
that, although an algebra of finite global dimension may not be
$(m,n)$-almost hereditary for any $m,n$, this property becomes true
for the algebra obtained from it by a (co)tilting using a stair
splitting (co)tilting module.
\begin{lem}
  \label{sec:m-n-quasitilted-4}
  Let $d$ be a positive integer. Let $A$ be an algebra with $\gd A = d$, $T$ an $A$-module and $B= (\End_A T)^{\rm op}$. 
  \begin{enumerate}
  \item[{\it (i)}] If $T$ is a stair splitting tilting module, then $B$ is $(d,1)$-almost hereditary.
  \item[{\it (ii)}] If $T$ is a stair splitting cotilting module, then $B$ is $(1,d)$-almost hereditary.
  \end{enumerate}
  In particular, any $(m,n)$-quasitilted algebra is $(m+n-1,1)$-almost hereditary, or else
  $(1,m+n-1)$-almost hereditary.
\end{lem}
\begin{proof}
  {\it (i)} Since $\di_A\,{\cal F}(T) \leq 1$ we
  can use the same considerations as in the proof of
  Theorem~\ref{sec:m-n-quasitilted} and conclude that $\di_B\,{\cal X}(T)
  \leq 1$. Also, $\dimp_B\,{\cal Y}(T) \leq d$ because $\dimgl A = d$.

  Now, assume that $B$ is $(m,n)$-quasitilted. Let $(A,T,B)$ be the
  resulting last triple appearing in the definition of
  $(m,n)$-quasitilted algebras. Then $\gd A=m+n-1$, and hence $B$ is
  $(m+n-1,1)$- or $(1,m+n-1)$-quasitilted according to whether $T$ is a
  tilting or cotilting module, respectively.
\end{proof}
From now on we focus on $(m,1)$-almost hereditary algebras. The
$(1,m)$-almost hereditary algebras may be treated using dual considerations.

\section{$(m,1)$-Almost hereditary algebras}
\label{sec:m-1-almost}
The purpose of this section is to prove (\ref{eq:14})
whenever $A$ is $(m,1)$-almost hereditary, (see
Theorem~\ref{principal}).  In order to prove this theorem, preparatory
material is first established in \ref{sec:algebras-with-small} on
algebras with nice small homological properties related to $(Q1)$ and
$(Q2)$. And the theorem is proved with some consequences in \ref{sec:applications-m-1}.

\subsection{On algebras with small homological dimensions}
\label{sec:algebras-with-small}
The proof of (\ref{eq:14}) when $A$ is $(m,1)$-almost hereditary is
mainly based on the fact that ${\cal L}_A^m$ is a torsion-free class
of $A$, which is true because $\dimgl\, A=m+1$. Hence, this subsection
is devoted to the investigation of (\ref{eq:14}) in the more general
situation where
\begin{itemize}
\item the class of $A$-modules with projective dimension at most $m$
  is replaced by a torsion-free class $\mathcal C$ of $\md A$;
\item ${\cal L}_A^m$ is replaced by the class of indecomposable $A$-modules such that
  all the predecessors in $\ind A$ of which lie in $\mathcal C$, this
  class is denoted by $\mathcal L_{\mathcal C}$.
\end{itemize}
For a given positive integer $n$, this investigation establishes
(\ref{eq:13}) whenever every indecomposable $A$-module lies in
$\mathcal C$ or else has injective dimension at most $n$ (see
Proposition~\ref{sec:suggestion-section-4-5}).
The first step of this investigation is to show that there are no
nonzero morphisms from any indecomposable $A$-module not lying in
$\mathcal C$ to any indecomposable $A$-module with a large injective
dimension. This is done in the two following lemmas.
\begin{lem}
  \label{sec:suggestion-section-4-1}
  Let $\mathcal C$
  be a torsion-free class of $\md A$. Consider a nonzero morphism
  $f\colon U \to V$, where $U$ and $V$ are indecomposable modules such that
  $U\not\in \mathcal C$ and $V\in \mathcal C$, and such that $\ell(U)+\ell(V)$
  is minimal for these properties. Then $\im f\in \mathcal C$,
  $\ker f\not\in \mathcal C$ and $\ker f$ is indecomposable.
\end{lem}
\begin{proof}
  Since $V\in \mathcal C$ and $\mathcal C$ is stable under taking
  submodules, it follows that
  \begin{equation}
    \label{eq:1}
    \im f\in \mathcal C.
  \end{equation}
  Using the exact sequence
  $0\to \ker f\to U\to \im f\to 0$, since $U\not\in \mathcal C$,
  $\im f\in \mathcal C$ and $\mathcal C$ is stable under extensions,
  then
  \begin{equation}
    \label{eq:2}
    \ker f\not\in \mathcal C \ \text{and $f$ is not a monomorphism.}
  \end{equation}
  Therefore, there exists an indecomposable direct summand $K$ of
  $\ker f$ such that $K\not\in \mathcal C$. Consider the push-out
  diagram
  \[
   \begin{tikzcd}
     0 \ar[r] & \ker f \ar[d] \ar[r] & U \ar[r] \ar[d] & \im f
     \ar[r] \ar{d}{1} & 0 \\
     0 \ar[r] & K \ar[r] & U'\ar[r] &  \im f \ar[r] & 0.
   \end{tikzcd}
  \]
  Notice that $U' \not\in \mathcal C$, because $K\not\in \mathcal C$ and $\mathcal C$ is closed under submodules.
  In this case, there exists an indecomposable direct summand $U''$
  of $U'$ such that $U''\not\in \mathcal C$. Consider the composite
  morphism
  \[
  U'' \to \im f \hookrightarrow V.
  \]
  This morphism is nonzero for the following reasons.
  Should the exact sequence $0 \to K \to U'\to \im f\to 0$ split,
  then the split epimorphism $\ker f \twoheadrightarrow K$ would
  factor through $U$, and hence there would exist a split epimorphism
  $U\to K$, which would entail that $\ker f=U=K$, a contradiction to
  $f$ being nonzero; therefore the exact sequence
  $0 \to K \to U'\to \im f\to 0$ does not split; accordingly, the
  coordinate morphism $U''\to \im f$ is nonzero, and hence nor is
  $U''\to V$.

  By construction, $\ell(U'')+\ell(V)\leqslant \ell(U)+\ell(V)$. By
  minimality of $\ell(U)+\ell(V)$ and because $U''\not\in \mathcal C$, it follows
  that $\ell(U'')=\ell(U)$. Accordingly, $U''=U'=U$, and hence
  $K=\ker f$. Thus $\ker f$ is indecomposable.
\end{proof}

\begin{lem}
  \label{sec:suggestion-section-4-2}
  Let $n$ be a positive integer. Let $\mathcal C$ be a torsion-free
  class of $\md A$ such that, for all $X\in \ind A$,
  \begin{center}
    $X\in \mathcal C$, or else $\di_A X\leqslant n$.
  \end{center}
  If $U$, $V\in \ind A$ such that $U\not\in \mathcal C$ and
  $\di_A V>n$, then $\hom_A(U,V)=0$.
\end{lem}
\begin{proof}
  Let $U$, $V\in \ind A$ be such that $U \not \in \mathcal C$ and
  $\di_A V>n$. By absurd, suppose that there exists a nonzero morphism
  $f \colon U\to V$. Assume that $\ell(U)+\ell(V)$ is minimal for these
  properties.

  \medskip

  According to Lemma~\ref{sec:suggestion-section-4-1}, $\im f\in \mathcal C$, $\ker f\not\in \mathcal C$ and
  $\ker f$ is indecomposable. Consequently,
  $\di_A \ker f\leqslant n$.

  \medskip

  Next, considering the injective dimensions of the modules in the
  following exact sequences
  \[
   \begin{tikzcd}
    0 \ar[r] & \ker f \ar[d]\ar[r] & U \ar[d]\ar[r] & \im f  \ar[d]\ar[r] & 0 \\
    0 \ar[r] & \im f \ar[r] & V \ar[r] & \coker f \ar[r] & 0
   \end{tikzcd}
  \]
  yields that $\di_A \im f\leqslant n$ and $\di_A \coker f>n$. This
  permits to prove that $\coker f$ is indecomposable. Indeed, the
  dual considerations of those following \eqref{eq:2} in the proof of
  Lemma~\ref{sec:suggestion-section-4-1} may be applied here to the
  exact sequence
  \[
  0 \rar \im f \rar V \rar \coker f \rar 0
  \]
  instead of to the exact sequence $0\rar \ker f\rar U \rar \im f \rar 0$
  provided that $\mathcal C$ is replaced by $\{M \in \md A;
  \di_A M\leqslant n\}$, which is stable under extensions. Accordingly,
  $\coker f$ is indecomposable.

  \medskip

  Finally, consider any exact sequence
  \[
  0 \to \coker f \to I_0 \to \cdots \to I_{n-2} \to C \to 0,
  \]
  where $I_0,\ldots,I_{n-2}$ are injective modules and $C$ is the $(n-1)$-th
  cosyzygy of $\coker f$; in the particular case $n=1$, just take
  $C=\coker f$ and discard the exact sequence. From the long exact
  sequence obtained upon applying $\hom_A(C,-)$ to the exact sequence
  \[
  0 \to \ker f \to U \to \im f \to 0,
  \]
  there results an exact sequence (recall that $\di_A \ker f\leqslant
  n$)
  \[
  \ext^n_A(C,U) \to \ext^n_A(C,\im f) \to 0.
  \]
  Accordingly, there exist $N,N_0,\ldots,N_{n-2}\in \md A$ fitting
  into a commutative diagram as follows, where the rows are exact and
  the leftmost square is cocartesian
  \[
   \begin{tikzcd}
    0 \ar[r] & U \ar[r] \ar[d] & N \ar[r] \ar[d] & N_0 \ar[r] \ar[d] &
    \cdots \ar[r] & N_{n-2} \ar[r] \ar[d] & C \ar[r] \ar[d]^{1} & 0 \\
    0 \ar[r] & \im f \ar[r] & V \ar[r]  & I_0 \ar[r]  &
    \cdots \ar[r] & I_{n-2} \ar[r]  & C \ar[r] & 0;
   \end{tikzcd}
  \]
  in the particular case $n=1$, take the lower row equal to the short
  exact sequence $0\rar \im f\rar V \rar C=\coker f\rar 0$ and take the
  sequence $N_0,\ldots,N_{n-2}$ to be void.

  Consider the following commutative diagram with exact
  rows and columns, whatever the value of $n$,
  \[
   \begin{tikzcd}
    & 0 \ar[d] & 0 \ar[d]\\
    & \ker f \ar[r]^{1} \ar[d] & \ker f \ar[d]\\
    0 \ar[r] & U \ar[r] \ar[d] & N \ar[r] \ar[d] & \coker f \ar[r]
    \ar[d]^{1} & 0 \\
    0 \ar[r] & \im f \ar[r] \ar[d] & V \ar[r] \ar[d] & \coker f \ar[r]
    & 0 \\
    & 0 & 0. & &
   \end{tikzcd}
  \]
  Notice that the exact sequence
  \begin{equation}
    \label{eq:3}
    0 \rar \ker f \rar N \rar V \rar 0
  \end{equation}
  does not split since, otherwise, the composite morphism $\ker f \to U
  \to N$ would be a section and hence the indecomposable $U$ would equal
  $\ker f$, a contradiction to $f$ being nonzero; similarly, the exact
  sequence
  \begin{equation}
    \label{eq:4}
    0 \rar U \rar N \rar \coker f\rar 0
  \end{equation}
  does not split.

  Now, consider the exact sequence
  \[
  0 \rar U \rar \im\,f \oplus N \rar V \rar 0.
  \]
  On one hand, since $\mathcal C$ is stable under taking submodules and
  $U\not \in  \mathcal C$, it follows that $\im f \oplus N \not \in \mathcal C$; and
  since $\im f\in \mathcal C$, there exists an indecomposable direct
  summand $N'$ of $N$ such that $N'\not\in \mathcal C$. On the other
  hand, because $\di_A\,U\leqslant n$ and $\di_A\,V>n$, it follows that
  $\di_A \im f \oplus N>n$; and since $\di_A \im f\leqslant n$, there exists
  an indecomposable direct summand $N''$ of $N$ such that
  $\di_A N''>n$.

  Because the exact sequences \eqref{eq:3} and \eqref{eq:4} do not split
  and $\ker f$ and $\coker f$ are indecomposable modules, the coordinate
  morphisms with indecomposable domain and codomain
  \[
  N'\rar V\ \text{and}\ U \rar N''
  \]
  are nonzero. Moreover, they feature the following properties,
  \begin{itemize}
  \item $N'\not\in \mathcal C$ and $\di_A\,V>n$; and
  \item $U \not \in \mathcal C$ and $\di_A\,N''>n$.
  \end{itemize}

  By assumption on $\mathcal C$, the indecomposable modules $N'$ and $N''$ are
  not isomorphic, and hence $N'\oplus N''$ is a direct summand of $N$;
  accordingly,
  \[
  \ell(N')+\ell(N'')\leqslant \ell(N) < \ell(U)+\ell(V).
  \]
  Therefore, $\ell(N')<\ell(U)$ or else $\ell(N'')<\ell(V)$. In the
  former case, $\ell(N')+\ell(V)<\ell(U)+\ell(V)$; and, in the
  latter case, $\ell(U)+ \ell(N'')<\ell(U)+\ell(V)$. Both cases
  contradict the minimality of $\ell(U)+\ell(V)$.
\end{proof}

The second step in proving (\ref{eq:13}) consists in gathering
information on indecomposable $A$-modules which have successors in
$\ind A$ with a large injective dimension. This is done in the
following lemma.
\begin{lem}
  \label{sec:suggestion-section-4-3}
  Let $n$ be a positive integer. Let $\mathcal C$ be a torsion-free
  class in $\md A$ such that, for all $X\in \ind A$,
  \begin{center}
    $X\in \mathcal C$, or else $\di_A X\leqslant n$.
  \end{center}
  Then
  \begin{enumerate}
  \item for every path $Y\to U\to V$ in $\ind A$ such that
    $\di_A V>n$, there exists $Z\in \ind A$ such that $\di_A Z>n$
    and $\hom_A(Y,Z)\neq 0$;
  \item for every $Y\in \ind A$ which has a successor in $\ind A$
    with injective dimension greater than $n$, there exists $Z\in \ind
    A$ such that $\di_A Z>n$ and $\hom_A(Y,Z)\neq 0$;
  \item if $X\in \ind A$ is such that $\di_A X>n$, then every
    predecessor of $X$ in $\ind A$ lies in $\mathcal C$.
  \end{enumerate}
\end{lem}
\begin{proof}
  (1) Let $Y\to U \to V$ be a path in $\ind A$ such that
  $\di_A V>n$. By absurd, assume that $\hom_A(Y,Z)=0$ for all $Z\in \ind A$ such
  that $\di_A Z>n$. Denote by $\mathcal C'$ the
  following torsion-free class of $\md A$
  \[
  \mathcal C' = \{M \in \md A; \hom_A(Y,M)=0\}.
  \]
  By assumption, $U\not \in \mathcal C'$, $\di_A V>n$ and the
  following holds for all $X\in \ind A$,
  \begin{center}
    $X\in \mathcal C'$, or else $\di_A X\leqslant n$.
  \end{center}
  Apply Lemma~\ref{sec:suggestion-section-4-2} to $U$ and $V$ for
  the torsion-free class $\mathcal C'$; then $\hom_A(U,V)=0$, a
  contradiction to $U\rar V$ being nonzero. Thus, there exists $Z\in
  \ind A$ such that $\di_A Z>n$ and $\hom_A(Y,Z)\neq 0$.

  \medskip

  (2) Using (1), an induction on $\ell$ shows that, for every
  $Y\in \ind A$, if there exists a path of length $\ell$ in $\ind A$
  starting in $Y$ and ending in a module with injective dimension
  greater than $n$, then there exists $Z\in \ind A$ such that
  $\di_A Z>n$ and $\hom_A(Y,Z)\neq 0$.

  \medskip

  (3) Let $X\in \ind A$ be such that $\di_A X >n$. Let $Y$ be a
  predecessor of $X$ in $\ind A$. Following (2), there exists $Z\in
  \ind A$ such that $\di_A Z>n$ and $\hom_A(Y,Z)\neq
  0$. Lemma~\ref{sec:suggestion-section-4-2} entails that $Y\in
  \mathcal C$.
\end{proof}

Now, it is possible to prove (\ref{eq:13}).
Recall that, for every torsion-free class $\mathcal C$ of $\md A$,
the piece of notation $\mathcal L_{\mathcal C}$ denotes the class of
$X\in \ind A$ such that every predecessor of $X$ in $\ind A$ lies in
$\mathcal C$.
\begin{prop}
  \label{sec:suggestion-section-4-5}
  Let $n$ be a positive integer. Let $\mathcal C$ be a torsion-free
  class of $\md A$ such that, for all $X\in \ind A$,
  \begin{center}
    $X\in \mathcal C$, or else $\di_A X\leqslant n$.
  \end{center}
  Then $\ind A = \mathcal L_{\mathcal C}\cup \mathcal R^n_A$.
\end{prop}
\begin{proof}
  Let $X\in \ind A \backslash \mathcal R^n_A$. Then there exists a
  successor $Z$ of $X$ in $\ind A$ such that $\di_A Z>n$. Part (3)
  of Lemma~\ref{sec:suggestion-section-4-3} entails that $Z\in
  \mathcal L_{\mathcal C}$. Thus $X\in \mathcal L_{\mathcal C}$.
\end{proof}

The previous proposition does not apply to $(m,n)$-almost hereditary
algebras for general $m$ and $n$ because, when $\dimgl A=m+n$, 
the class of $A$-modules with projective dimension at most $m$ need
not be torsion-free. However,
Proposition~\ref{sec:suggestion-section-4-5} may be applied to
certain algebras satisfying $(Q2)$ as the following result shows.
\begin{thm}
  \label{sec:suggestion-section-4-4}
  Let $m$, $n$ be positive integers. Assume that
  \begin{enumerate}[(a)]
  \item $\dimgl A = {\rm max}\{m,n\}+1$; and
  \item for all $X\in \ind A$, then $\dimp_A X\leqslant m$, or else
    $\di_A X\leqslant n$.
  \end{enumerate}
  Then $\ind A = \mathcal L^m_A \cup \mathcal R^n_A$.
\end{thm}
\begin{proof}
  Assume first that $m\geqslant n$. Denote by $\mathcal C$ the class
  \[
  \mathcal C = \{M \in \md A; \dimp_A M\leqslant m\};
  \]
  since $\dimgl A = m+1$, this is a torsion-free class in
  $\md A$. By definition,
  $\mathcal L_A^m = \mathcal L_{\mathcal C}$. Hence, the conclusion
  follows from Proposition~\ref{sec:suggestion-section-4-5}.

  When $n\geqslant m$, the conclusion follows from dual
  considerations, using the dual version of
  Proposition~\ref{sec:suggestion-section-4-5}.
\end{proof}

\subsection{Applications to $(m,1)$-almost hereditary algebras}
\label{sec:applications-m-1}

Now, we can prove the main result of this section.
\begin{thm}
  \label{principal}
  Let $A$ be a $(m,1)$-almost hereditary algebra. Then $\ind A = {\cal L}_A^m \cup {\cal R}_A$.
\end{thm}
\begin{proof}
  The theorem now
  follows from Theorem~\ref{sec:suggestion-section-4-4}.
\end{proof}

It is a consequence of Theorem \ref{principal} that the subcategories
${\cal L}_A^m$ and ${\cal R}_A$ determine a trisection in the sense of
\cite[Chapter II, Section 1, p. 36]{MR1327209} in $\ind A$, when $A$ is a $(m,1)$-almost
hereditary algebra, as stated in the following corollary.

\begin{cor} Let $A$ be an algebra which satisfies conditions {\it (i)} and {\it (ii)} as described in Theorem~\ref{sec:suggestion-section-4-4}. Then $({\cal L}_A^m \setminus {\cal R}_A^n, {\cal L}_A^m \cap {\cal R}_A^n, {\cal R}_A^n \setminus {\cal L}_A^m)$ is a trisection in $\ind A$. In particular, if $A$ is $(m,1)$-almost hereditary, then $({\cal L}_A^m \setminus {\cal R}_A, {\cal L}_A^m \cap {\cal R}_A, {\cal R}_A \setminus {\cal L}_A^m)$ is a trisection in $\ind A$.
\end{cor} 

To end this section, we discuss sufficient conditions for an algebra
to be $(m,1)$-almost hereditary. It is proved in \cite{MR1327209},
then an algebra $A$ is $(1,1)$-almost hereditary if and only if all
indecomposable projective modules belong to ${\cal L}_A$. When
replacing $\cal L_A$ and $(1,1)$ by $\cal L_A^m$ and $(m,1)$,
respectively, part of the equivalence may be proved. In order to do
so, it is again convenient to first replace the class of modules with
projective dimension at most $m$ by a torsion-free class.
\begin{prop}
  \label{sec:suggestion-section-4-6}
  Let $\mathcal C$ be a torsion-free
  class of $\md A$. If $A\in {\rm add}\,\mathcal L_{\mathcal C}$,
  then for all $X\in \ind A$,
  \begin{center}
    $X\in \mathcal C$, or else $\di_A X\leqslant 1$.
  \end{center}
\end{prop}
\begin{proof}
  Let $X\in \ind A$ be of injective dimension at least $2$. Then
  $\hom_A(\tau^{-1}X,A)\neq 0$. Therefore, $\tau^{-1}X\in \mathcal
  L_{\mathcal C}$, and hence $X\in \mathcal C$.
\end{proof}

Now, here is a sufficient condition for an algebra to be
$(m,1)$-almost hereditary.
\begin{prop}
  \label{sec:suggestion-section-4-7}
  Let $m$ be a positive integer. Assume that
  $A\in {\rm add}\,\mathcal L^m_A$. Then
  \begin{enumerate}
  \item $\dimgl A\leqslant m+1$ and, for all $X\in \ind A$,
    then $\dimp_A X\leqslant m$ or else $\di_A X\leqslant 1$;
  \item if, moreover, $\dimgl A=m+1$, then $A$ is
    $(m,1)$-almost hereditary.
  \end{enumerate}
\end{prop}
\begin{proof}
  (1) Since $A\in \mathcal L^m_A$, the syzygy of any $A$-module has
  projective dimension at most $m$. Accordingly, ${\rm
    gl.dim}\,A\leqslant m+1$. Denote by $\mathcal C$ the class
  \[
  \mathcal C = \{M \in \md A\ |\ \dimp_A\,M\leqslant m\}\,;
  \]
  since ${\rm gl.dim}\,A\leqslant m+1$, this is a torsion-free class of
  $\md A$. The rest of the statement of (1) therefore follows from
  Proposition~\ref{sec:suggestion-section-4-6}.

  \medskip

  (2) follows directly from (1).
\end{proof}

We make the following conjecture.
\begin{conj}
  Any $(m,1)$-almost hereditary algebra $A$ is
  such that $A \in \add \ {\cal L}_A^m$.  
\end{conj}
The corresponding statement for $(m,n)$-almost hereditary does not
hold when $n>1$ as shown in the following example.
\begin{ex}
  Let $A$ be the
  radical square zero path algebra given by the quiver $m+n+1 \rightarrow \cdots \rightarrow 2 \rightarrow 1.$
  
  Then $A$ is
  $(m,n)$-almost hereditary and $P_{m+n+1} \not \in {\cal L}_A^m$.
\end{ex}

\section{One-point extensions of $(m,1)$-almost hereditary algebras}
\label{sec:one-point-extensions}
From now on, we assume that $k$ is a field and $A$ is a finite
dimensional $k$-algebra. The purpose of this section is to investigate
how $(m,1)$-almost hereditary algebras behave under one-point
extension process. In particular, conditions for the one-point
extension to be $(m,1)$-almost hereditary are presented. First, here
is a necessary condition for a one-point extension to be a
$(m,1)$-almost hereditary algebra.
\begin{prop}
  \label{AentaoB}
  Assume that $k$ is a field. Let $B$ be a finite dimensional
  $k$-algebra with $\gd B = m+1$ and assume that $A = B[M]$ for some
  $B$-module. If $A$ is a $(m,1)$-almost hereditary algebra, then $B$ is
  $(m,1)$-almost hereditary.
\end{prop}
\begin{proof}
  When $m=1$, the proposition is proved in \cite[Chapter III,
  Proposition 2.3]{MR1327209}. Up to replacing the inequality
  ``$\dimp>1$'' by ``$\dimp>m$'', the proof given there works here in the
  general case.
\end{proof}

We point out that the hypothesis on $\gd B$ above is necessary to
conclude that $B$ is $(m,1)$-almost hereditary for some positive
integer $m$ as shown in the following example.
\begin{ex}
  Let $B$ be the path algebra of the
  quiver
   \begin{center}
     \begin{tikzcd}
       5 \ar{r} \arrow[rr, bend left, dashed, dash] & 4 \ar{r} & 3 \ar{r} \arrow[rr, bend left, dashed, dash] & 2 \ar{r} & 1,
     \end{tikzcd}
   \end{center}
  which has $\gd B = 2$ but is not $(1,1)$-almost hereditary. However,
  by letting $M = S_5$ we obtain the $(2,1)$-almost hereditary algebra
  $A = B[M]$ given by the quiver
   \begin{center}
     \begin{tikzcd}
       6 \ar[r] \arrow[rr, bend left, dashed, dash] & 5 \ar[r] \arrow[rr, bend right, dashed, dash] & 4 \ar[r] & 3 \ar[r] \arrow[rr, bend left, dashed, dash] & 2 \ar[r] & 1.
     \end{tikzcd}
   \end{center}
\end{ex}

For $m = 1$, it is well known that the converse of Proposition \ref{AentaoB} does not hold true. This is also the case for each positive integer $m \ne 1$, as we can see in the following example.

\begin{ex}
  \label{not(m,1)}
  Let $B$ be the path algebra given by the quiver
   \begin{center}
     \begin{tikzcd}[column sep=small]
       m+3  \ar{r} & m+2 \arrow[rr, bend left, dashed, dash]   \ar{r} & m+1 \arrow[rr, bend right, dashed, dash] \ar{r} & m \ar{r} & \cdots \ar{r}\arrow[rr, bend right, dashed, dash] & 3 \ar{r} \arrow[rr, bend left, dashed, dash] & 2 \ar{r} & 1. 
     \end{tikzcd}
   \end{center}
  where all the paths of length two with source in $\{1,\ldots,m+2\}$ are relations.
  It can be easily checked that $B$ is $(m,1)$-almost hereditary. Now,
  taking $M = I_{m+3}=S_{m+3}$, the one-point extension algebra $A = B[M]$ is such that $\gd A = \gd B$, but it is not $(m,1)$-almost hereditary, since $\dimp_A S_{m+2} = m+1$ and $\di_A S_{m+2} = 2$. \\
\end{ex}

Besides Example~\ref{not(m,1)}, if $M \in \add  \ {\cal L}_B^m$ then property $(Q2)$ of the definition of $(m,n)$-almost hereditary algebra is satisfied by those indecomposable $B[M]$-modules of the shape $(0,X,0)$, as stated in the following proposition.

\begin{prop}
  Assume that $k$ is a field. Let $B$ be a $(m,1)$-almost hereditary
  algebra oer $k$ and let $A = B[M]$ for some $B$-module $M$. If
  $M \in \add \ {\cal L}^m_B$, then:
  \begin{enumerate}
  \item[{\it (i)}] $\gd A = m+1$,
  \item[{\it (ii)}] if $(0,X,0)$ is an indecomposable $A$-module, then
    $\dimp_A (0,X,0) \leq m$, or else $\di_A (0,X,0) \leq 1$.
  \end{enumerate}
\end{prop}
\begin{proof}
  When $m=1$, the proposition is proved in \cite[Chapter III, Lemma
  2.5]{MR1327209}. The proof given there may be adapted to prove the
  general case by replacing the inequality ``$\dimp>1$'' by ``$\dimp>m$''.
\end{proof}

In order to present a sufficient condition for $B = A[M]$ to be a $(m,1)$-almost hereditary algebra, the following technical lemma is needed.

\begin{lem}
  \label{pdextension}
  Assume that $k$ is a field. Let $B$ be a finite dimensional $k$-algebra
  with $\gd B = m+1$ and assume that  $A = B[M]$ for a $B$-module $M$. If
  $(Y, X, f)$ is an $A$-module, then $\dimp_A (Y, X, f) \leq m$ if and
  only if:
  \begin{enumerate}
  \item[$(1)$] $\dimp_B \ker f \leq m-1$,
  \item[$(2)$] $\ext_B^{m-1}(\ker f, -) \xrar{\theta} \ext_B^{m+1}(\coker f, -) \rar 0$ is an exact sequence, where $\theta$ is naturally induced by $f$.
  \end{enumerate} 
\end{lem}
\begin{proof}
  Let $(Y, X, f)$ be an $A$-module. Since $(0, P_{\coker f}, 0) \oplus (Y, M \otimes_k Y, 1_{M\otimes Y})$ is a projective cover of $(Y, X, f)$, it is possible to construct the following commutative diagram
   \begin{center}
     \begin{tikzcd}
       0 \ar[r] & 0 \ar[r] \ar{d} & M \otimes_k Y \ar{r}{1_{M \otimes Y}} \ar{d}{\left( \begin{smallmatrix} 1 \\ 0 \end{smallmatrix}\right)} & M \otimes_k Y \ar{r} \ar{d}{f} & 0 \\
       0 \ar{r} & K \ar{r} & M \otimes_k Y \oplus P_{\coker f} \ar{r} & X \ar{r} & 0. 
     \end{tikzcd}
   \end{center}
  In this case, $\dimp_A (Y, X, f) \leq m$ if and only if
  $\dimp_A (0,K,0) \leq m-1$. By the snake Lemma we get the exact
  sequences
  $0 \rar \ker f \rar K \rar P_{\coker f} \rar \coker f \rar 0$ and
  $0 \rar \ker f \rar K \rar \Omega^1(\coker f) \rar 0$, and from the
  latter we obtain that
  \begin{equation}
    \label{eq:5}
    \dimp_B \ker f \leq \maximo \{ \dimp_B K, \dimp_B \Omega^1(\coker f)
    - 1\}\,.
  \end{equation}

  Since $\gd B = m+1$, it follows from (\ref{eq:5}) that $\dimp_B \ker f
  \leq m-1$ if $\dimp_A (Y,X,f) \leq m$, which proves $(1)$. Now, by
  applying $\hom(\_,-)$ to $0 \rar \ker f \rar K \rar \Omega^1(\coker f)
  \rar 0$ we get the exact sequence $\ext_B^{m-1}(\ker f, -) \xrar{g}
  \ext_B^{m}(\Omega^1(\coker f), -) \rar 0$. The connecting
  morphism $\delta: \ext_B^m(\Omega^1(\coker f), -) \rar
  \ext_B^{m+1}(\coker f, -)$ is an isomorphism, hence
  $\ext_B^{m-1}(\ker f, -) \xrar{\theta} \ext_B^{m+1}(\coker f, -) \rar
  0$ is an exact sequence, where $\theta = \delta g$. 

  Conversely, assume (1) and (2). In order to prove that
  $\dimp_A(Y,X,f)\leqslant m$, it is enough to prove that $\dimp_B K \leq m-1$. Applying $\hom(\_,-)$ to $$0 \rar \ker f \rar K \rar \Omega^1(\coker f) \rar 0$$ gives the exact sequence $\ext_B^{m-1}(\ker f, -) \xrar{g} \ext_B^m(\Omega^1(\coker f), -) \rar \ext_B^m(K, -) \rar \ext_B^m(\ker f, -) = 0.$ It follows from $(2)$ that $\theta = \delta g$ is an epimorphism, and so is $g$. Hence $\ext_B^m(K, -) = 0$, 
  that is, $\dimp_B K \leq m-1$.
\end{proof}

Now, here is a sufficient condition for a one-point extension to be
$(m,1)$-almost hereditary.
\begin{thm}
  \label{sec:one-point-extensions-1}
  Assume that $k$ is a field. Let $B$ be a $(m,1)$-almost hereditary
  algebra and assume that $A = B[M]$ for a projective $B$-module $M$. If, for
  every indecomposable $A$-module $(k^t, X, f)$, one has
  $\dimp_B X \leq m$ or else $\di_B X \leq 1$, then $A$ is
  $(m,1)$-almost hereditary.
\end{thm}
\begin{proof}
  It is proved in \cite[Chapter III, Proposition 1.3]{MR1327209} that
  $\gd A = \maximo\{\gd B, \dimp_B M + 1\}$, which allows us to
  conclude that $\gd A = m+1$. Let $(k^t, X, f)$ be an indecomposable
  $A$-module. Assume first $t=0$. In this case, $\dimp_B X \leq m$ or
  else $\di_B X \leq 1$, since $B$ is $(m,1)$-almost hereditary and
  $X$ is an indecomposable $B$-module. Therefore
  $\dimp_A (0,X,0) \leq m$ or else $\di_A (0,X,0) \leq 1$.

  Suppose now $t \ne 0$. According to \cite[Chapter III, Lemma
  2.2]{MR1327209}, if $\di_B X \leq 1$, then $\di_A (k^t, X, f) \leq 1$,
  because $\ext_B^1(M,X) =0$. Finally, assume that $\dimp_B X \leq m$
  and consider the following short exact sequences
  \begin{equation}
    \label{eq:6}
    \left\{
      \begin{array}{l}
        0 \rar \ker f \rar M\otimes_k k^t \rar \im f \rar 0\\
        0 \rar \im f \rar X \rar \coker f \rar 0\,.
      \end{array}\right.
  \end{equation}

  \noindent Lemma~\ref{pdextension} will be applied to conclude that
  $\dimp_A (k^t, X, f) \leq m$. For this purpose we first prove that
  $\dimp_B \ker f \leq m-1$. Applying $\hom(\_, -)$ to (\ref{eq:6})
  gives rise to the exact sequences
  \begin{equation}
    \left\{
      \begin{array}{l}
        0 = \ext_B^m(M^t, -) \rar \ext_B^m(\ker f, -) \rar \ext_B^{m+1}(\im
        f, -) \rar \ext_B^{m+1}(M^t, -) = 0 \\
        0 = \ext_B^{m+1}(X, -) \rar \ext_B^{m+1}(\im f, -) \rar
        \ext_B^{m+2}(\coker f, -) = 0\,,
      \end{array}\right.
  \end{equation}
  from where we get that $\ext_B^{m}(\ker f, -) = 0$. Thus $\dimp_B \ker f \leq m-1$.

  Now, since $\dimp_B X \leq m$ there are exact sequences
  $\ext_B^{m-1}(\ker f, -) \xrar{\alpha} \ext_B^{m}(\im f, -) \rar 0$
  and $0 \rar \ext_B^m(\im f, -) \xrar{\beta} \ext_B^{m+1}(\coker f, -)
  \rar \ext_B^{m+1}(X, -) =0$. Therefore 
  \begin{center}
    $\ext_B^{m-1}(\ker f, -) \xrar{\beta \alpha} \ext_B^{m+1}(\coker f, -) \rar 0$
  \end{center}
  is an exact sequence. By Lemma \ref{pdextension}, $\dimp_A (k^t, X, f)
  \leq m$, which finishes the proof.
\end{proof}

\begin{cor}
  Assume that $k$ is a field. Let $B$ be a $(m,1)$-almost hereditary
  algebra and let $M$ be a projective $B$-module such that
  $$\hom_B(M, -)|_{{\cal R}_B \setminus {\cal L}_B^m} = 0.$$ Then
  $B[M]$ is $(m,1)$-almost hereditary.
\end{cor}
\begin{proof}
  Let $t$ be a natural integer. Notice that if
  $(k^t, X, f)$ is an indecomposable $A$-module, then any indecomposable
  direct summand of $X$ is a successor in $\ind B$ of an
  indecomposable direct summand of $M^t$. Since there is no nonzero
  morphism from $M$ to ${\cal R}_B \setminus {\cal L}_B^m$ we
  get that $X \in \add  {\cal L}_B^m$, hence $\dimp_B X \leq m$. The
  conclusion therefore follows from Theorem~\ref{sec:one-point-extensions-1}.
\end{proof}


\end{document}